\documentclass[12pt, reqno]{amsart}

\usepackage{amsmath, amsthm, amscd, amsfonts, amssymb, graphicx, color}
\usepackage[bookmarksnumbered, colorlinks, plainpages]{hyperref}
\usepackage{mathtools}
\usepackage{enumerate}
\usepackage{amscd} 
\usepackage{xy}
\xyoption{all}

\newtheorem{theorem}{Theorem}[section]
\newtheorem{lemma}[theorem]{Lemma}
\newtheorem{proposition}[theorem]{Proposition}
\newtheorem{corollary}[theorem]{Corollary}
\theoremstyle{definition}
\newtheorem{definition}[theorem]{Definition}

\newtheorem{remark}[theorem]{Remark}

\begin{document}
\setcounter{page}{1}

\title[ ]{On Some Properties and Inequalities for the Nielsen's $\beta$-Function}

\author[K. Nantomah]{Kwara Nantomah }

\address{ Department of Mathematics, Faculty of Mathematical Sciences, University for Development Studies, Navrongo Campus, P. O. Box 24, Navrongo, UE/R, Ghana. }
\email{\textcolor[rgb]{0.00,0.00,0.84}{mykwarasoft@yahoo.com, knantomah@uds.edu.gh}}




\subjclass[2010]{33Bxx, 33B99, }

\keywords{Nielsen's $\beta$-function; digamma function; subadditive; superadditive; inequality }

\date{Received: xxxxxx; Revised: xxxxxx; Accepted: xxxxxx.
\newline \indent $^{*}$ Corresponding author}

\begin{abstract}
In this study, we obtain some convexity, monotonicity and additivity properties as well as some inequalities involving the Nielsen's $\beta$-function which was introduced in 1906. 
\end{abstract} \maketitle


\section{Introduction and Preliminaries}

\noindent
The Nielsen's $\beta$-function, $\beta(x)$  which was introduced in \cite{Nielsen-1906-Teubner} is defined as  
\begin{align}
\beta(x)&=\int_0^{1} \frac{t^{x-1}}{1+t}\,dt, \quad x>0   \label{eqn:Nielsen-funt-Integral-Rep-1} \\
&=\sum_{k=0}^{\infty}\frac{(-1)^k}{k+x},  \quad x>0   \label{eqn:Nielsen-funt-Series-Rep}
\end{align}
and by change of variables, the representation ~(\ref{eqn:Nielsen-funt-Integral-Rep-1}) can be written as 
\begin{equation}\label{eqn:Nielsen-funt-Integral-Rep-2}
\beta(x)=\int_0^{\infty} \frac{e^{-xt}}{1+e^{-t}}\,dt, \quad x>0 . 
\end{equation}
The function $\beta(x)$ is also defined as \cite{Nielsen-1906-Teubner}
\begin{equation}\label{eqn:Nielsen-funt-via-Digamma}
\beta(x)=\frac{1}{2}\left \{ \psi \left( \frac{x+1}{2}\right) - \psi \left( \frac{x}{2}\right) \right \} 
\end{equation}
where $\psi(x)=\frac{d}{dx}\ln \Gamma(x)=\frac{\Gamma'(x)}{\Gamma(x)}$ is the digamma function and $\Gamma(x)$ is the Euler's Gamma function. See also \cite{Boyadzhiev-Medina-Moll-2009-Scientia}, \cite{Connon-2012-arXiv}, \cite{Gradshteyn-Ryzhik-2014-AP} and  \cite{Medina-Moll-2009-Scientia}.  \\

\noindent
It is known that function $\beta(x)$ satisfies the following properties \cite{Boyadzhiev-Medina-Moll-2009-Scientia},\cite{Nielsen-1906-Teubner}.
\begin{align}
\beta(x+1)&=\frac{1}{x} - \beta(x),   \label{eqn:Nielsen-funct-Funct-Eqn} \\
\beta(x) + \beta(1-x)&=\frac{\pi}{\sin \pi x}, \label{eqn:Reflect-form-Nielsen-funct}.
\end{align}
In particular, $\beta(1)=\ln 2$, $\beta\left( \frac{1}{2}\right)=\frac{\pi}{2}$, $\beta\left( \frac{3}{2}\right)=2-\frac{\pi}{2}$ and $\beta(2)=1-\ln 2$.

\begin{proposition}
The function $\beta(x)$ is related to the classical Euler's beta function, $B(x,y)$ in the following ways.
\begin{equation}\label{eqn:Nielsen-funt-rel-Beta-1}
\beta(x) = - \frac{d}{dx}\left\{ \ln B\left(\frac{x}{2}, \frac{1}{2}\right) \right\}, 
\end{equation}


\begin{equation}\label{eqn:Nielsen-funt-rel-Beta-3}
\beta(x) + \beta(1-x) = B(x,1-x) .
\end{equation}

\end{proposition}

\begin{proof}
By the Euler's beta function $B(x,y)=\frac{\Gamma(x)\Gamma(y)}{\Gamma(x+y)}$, we obtain
\begin{equation}\label{eqn:Beta-Relation}
B\left( \frac{x}{2}, \frac{1}{2}\right)=\frac{\Gamma \left(\frac{x}{2}\right) \Gamma \left(\frac{1}{2}\right)}{\Gamma \left(\frac{x+1}{2}\right)}.
\end{equation}
Then by taking the logarithmic derivative of ~(\ref{eqn:Beta-Relation}) and using ~(\ref{eqn:Nielsen-funt-via-Digamma}), we obtain
\begin{align*}
\frac{d}{dx} \left\{ \ln B\left( \frac{x}{2}, \frac{1}{2}\right) \right\}
=\frac{1}{2} \frac{B'\left( \frac{x}{2}, \frac{1}{2}\right)}{B\left( \frac{x}{2}, \frac{1}{2}\right)}
&= \frac{1}{2} \left \{  \frac{\Gamma' \left(\frac{x}{2}\right)}{\Gamma \left(\frac{x}{2}\right)} - \frac{\Gamma' \left(\frac{x+1}{2}\right)}{\Gamma \left(\frac{x+1}{2}\right)} \right \}   \\
&= \frac{1}{2} \left \{ \psi \left(\frac{x}{2}\right) - \psi \left(\frac{x+1}{2}\right) \right \}  \\
&= - \beta(x) 
\end{align*}
yielding the result ~(\ref{eqn:Nielsen-funt-rel-Beta-1}). The result ~(\ref{eqn:Nielsen-funt-rel-Beta-3}) follows easily from the relation ~(\ref{eqn:Reflect-form-Nielsen-funct}).
\end{proof}

\begin{remark}
The function $\beta(x)$ is referred to as the \textit{incomplete beta function} in  \cite{Boyadzhiev-Medina-Moll-2009-Scientia} and \cite{Medina-Moll-2009-Scientia}. However, this should not be confused with the incomplete beta function which is usually defined as 
\begin{equation*}
B(a;x,y)=\int_{0}^{a}t^{x-1}(1-t)^{y-1}\,dt \quad x>0, y>0
\end{equation*}
or the regularized incomplete beta function which is defined as
\begin{equation*}
I_{a}(x,y)=\frac{B(a;x,y)}{B(x,y)}  \quad x>0, y>0.
\end{equation*}
Also, the function should not be confused with Dirichlet's beta function which is defined as \cite{Finch-2003-CUP}
\begin{equation*}
\beta^*(x)=\sum_{k=0}^{\infty}\frac{(-1)^k}{(2k+1)^x} = \frac{1}{\Gamma(x)}\int_{0}^{\infty}\frac{t^{x-1}}{e^{t}+e^{-t}}\,dt, \quad x>0 .
\end{equation*}
\end{remark}

\noindent
We shall use the notations $\mathbb{N}=\{1,2,3,\dots,\}$ and $\mathbb{N}_0=\mathbb{N} \cup \{0\}$ in the rest of the paper. \\

\noindent
By differentiating $m$ times of  ~(\ref{eqn:Nielsen-funt-Integral-Rep-1}), ~(\ref{eqn:Nielsen-funt-Series-Rep}) and ~(\ref{eqn:Nielsen-funt-Integral-Rep-2}), we obtain
\begin{align}
\beta^{(m)}(x)&=\int_0^{1} \frac{(\ln t)^{m}t^{x-1}}{1+t}\,dt, \quad x>0   \label{eqn:m-deriv-Nielsen-funt-Integral-Rep-1}\\
&=(-1)^{m}m!\sum_{k=0}^{\infty}\frac{(-1)^k}{(k+x)^{m+1}},  \quad x>0 \label{eqn:m-deriv-Nielsen-funt-Series-Rep} \\
&=(-1)^{m} \int_0^{\infty} \frac{t^{m}e^{-xt}}{1+e^{-t}}\,dt, \quad x>0 \label{eqn:m-deriv-Nielsen-funt-Integral-Rep-2} 
\end{align}
for $m\in \mathbb{N}_0$. It is clear that $\beta^{(0)}(x)=\beta(x)$. In particular, we have
\begin{align}
\beta^{(m)}(1)&=(-1)^{m}m!\sum_{k=0}^{\infty}\frac{(-1)^k}{(k+1)^{m+1}} = (-1)^{m}m! \eta(m+1), \quad m\in \mathbb{N}_0   \label{eqn:Nielsen-vrs-Eta-Funct} \\
&=(-1)^{m}m!\left(1-\frac{1}{2^m} \right) \zeta(m+1), \quad m\in \mathbb{N} \label{eqn:Nielsen-vrs-Zeta-Funct}
\end{align}
where $\eta(x)$ is the Dirichlet's eta function and $\zeta(x)$ is the Riemann zeta function defined as
\begin{equation*}
\eta(x)=\sum_{k=0}^{\infty}\frac{(-1)^k}{(k+1)^x}, \, x>0 \quad \text{and} \quad \zeta(x)=\sum_{k=1}^{\infty}\frac{1}{k^x}, \,  x>1.
\end{equation*}

\noindent
Then by differentiating $m$ times of  ~(\ref{eqn:Nielsen-funt-via-Digamma}) and ~(\ref{eqn:Nielsen-funct-Funct-Eqn}), we obtain respectively
\begin{equation}\label{eqn:Nielsen-funct-Gen-Funct-Eqn}
\beta^{(m)}(x+1)=\frac{(-1)^m m!}{x^{m+1}} - \beta^{(m)}(x) 
\end{equation}
and
\begin{equation}\label{eqn:Nielsen-funt-via-Digamma-mth-deriv}
\beta^{(m)}(x)=\frac{1}{2^{m+1}}\left \{ \psi^{(m)} \left( \frac{x+1}{2}\right) - \psi^{(m)} \left( \frac{x}{2}\right) \right \} .
\end{equation}

\noindent
For rational arguments $x=\frac{p}{q}$, the function $\psi^{(m)}(x)$ takes the form 
\begin{equation}\label{eqn:Polygamma-Rational-values}
\psi^{(m)}\left(\frac{p}{q}\right) = (-1)^{m+1}m!q^{m+1}\sum_{k=0}^{\infty}\frac{1}{(qk+p)^{m+1}}, \quad m\geq1
\end{equation}
which implies 
\begin{equation}\label{eqn:Polygamma-Rational-values-relation}
\psi^{(m)}\left(\frac{3}{4}\right) - \psi^{(m)}\left(\frac{1}{4}\right) = (-1)^{m+1}m!4^{m+1} \left\{ \sum_{k=0}^{\infty}\frac{1}{(4k+3)^{m+1}} - \sum_{k=0}^{\infty}\frac{1}{(4k+1)^{m+1}} \right\} . 
\end{equation}
Let $x=\frac{1}{2}$ in ~(\ref{eqn:Nielsen-funt-via-Digamma-mth-deriv}). Then we obtain
\begin{equation}\label{eqn:mth-deriv-Nielsen-func-of-half-1}
\beta^{(m)}\left(\frac{1}{2} \right)=\frac{1}{2^{m+1}}\left \{ \psi^{(m)} \left( \frac{3}{4}\right) - \psi^{(m)} \left( \frac{1}{4}\right) \right \} 
\end{equation}
which by ~(\ref{eqn:Polygamma-Rational-values-relation}) can be written as
\begin{equation}\label{eqn:mth-deriv-Nielsen-func-of-half-2}
\beta^{(m)}\left(\frac{1}{2} \right) = (-1)^{m+1}m!2^{m+1} \left\{ \sum_{k=0}^{\infty}\frac{1}{(4k+3)^{m+1}} - \sum_{k=0}^{\infty}\frac{1}{(4k+1)^{m+1}} \right\} .
\end{equation}
Now let $m=1$ in ~(\ref{eqn:mth-deriv-Nielsen-func-of-half-2}). Then we obtain
\begin{equation}\label{eqn:eqn:1st-deriv-Nielsen-func-of-half-A}
\beta' \left(\frac{1}{2} \right) = 4\left\{ \sum_{k=0}^{\infty}\frac{1}{(4k+3)^{2}} - \sum_{k=0}^{\infty}\frac{1}{(4k+1)^{2}} \right\} = -4G
\end{equation}
where $G=0.915 965 594 177...$ is the Catalan's constant.

\begin{remark}
The Catalan's constant has several interesting representations \cite{Bradley-2001-Online}, and amongst them are:
\begin{equation*}\label{eqn:Catalan-Constant-Pop-Rep}
G=\sum_{k=0}^{\infty}\frac{(-1)^k}{(2k+1)^2}, 
\end{equation*}
\begin{equation}\label{eqn:Catalan-Constant-Rep-1}
G=-\frac{\pi^2}{8}+2\sum_{k=0}^{\infty}\frac{1}{(4k+1)^2}, 
\end{equation}
\begin{equation}\label{eqn:Catalan-Constant-Rep-2}
G=\frac{\pi^2}{8}-2\sum_{k=0}^{\infty}\frac{1}{(4k+3)^2} .
\end{equation}
Thus, ~(\ref{eqn:eqn:1st-deriv-Nielsen-func-of-half-A}) is a consequence ~(\ref{eqn:Catalan-Constant-Rep-1})  and ~(\ref{eqn:Catalan-Constant-Rep-2}) .
\end{remark}

\noindent
 Equivalently, by letting $m=1$ in ~(\ref{eqn:mth-deriv-Nielsen-func-of-half-1}) we obtain
\begin{equation*}\label{eqn:eqn:1st-deriv-Nielsen-func-of-half-B}
\beta'\left(\frac{1}{2}\right)=\frac{1}{4}\left \{ \psi' \left( \frac{3}{4}\right) - \psi' \left( \frac{1}{4}\right) \right \}=-4G
\end{equation*}
since $ \psi' \left( \frac{1}{4}\right)=\pi^2 + 8G$ and $ \psi' \left( \frac{3}{4}\right)=\pi^2 - 8G$. See \cite{Boyadzhiev-Medina-Moll-2009-Scientia} and \cite{Kolbig-1996-JCAM}. 

\noindent
By using ~(\ref{eqn:Nielsen-vrs-Eta-Funct}), ~(\ref{eqn:Nielsen-vrs-Zeta-Funct}), ~(\ref{eqn:Nielsen-funct-Gen-Funct-Eqn}) and ~(\ref{eqn:eqn:1st-deriv-Nielsen-func-of-half-A}), we derive the following special values.
\begin{align*}
\beta'(1)&=-\frac{1}{2}\zeta(2)=-\frac{\pi^2}{12}, \\
\beta'(2)&=-1+\frac{\pi^2}{12}, \\
\beta'(3)&=\frac{3}{4} - \frac{\pi^2}{12},  \\
\beta'\left(\frac{3}{2}\right)&=4(G-1), \\
\beta'\left(\frac{5}{2}\right)&=\frac{40}{9}-4G.
\end{align*}
More special values may be derived by using similar procedures. As shown in \cite{Boyadzhiev-Medina-Moll-2009-Scientia} and \cite{Gradshteyn-Ryzhik-2014-AP}, the Nielsen's $\beta$-function  is very useful in evaluating certain integrals.

\section{Main Results}

\noindent
To start with, we recall the following well-known definitions.

\begin{definition}\label{def:Log-Convexity}
A function $f:I\rightarrow \mathbb{R}$ is said to be logarithmically convex if 
\begin{equation*}
\log f(ux + vy)\leq u \log f(x)+ v \log f(y)
\end{equation*}
or equivalently
\begin{equation*}
f(ux + vy) \leq (f(x))^{u}(f(y))^{v}
\end{equation*}
for each $x,y\in I$ and $u,v>0$ such that $u+v=1$. 
\end{definition}

\begin{definition}\label{def:complete-monotonicity}
A function $f: (0,\infty)\rightarrow R$ is said to be completely monotonic  if $f$ has derivatives of all order and
\begin{equation*}
(-1)^{k}f^{(k)}(x)\geq0 \quad \text{for} \quad x\in(0,\infty), \quad k\in \mathbb{N}_0 .
\end{equation*}
\end{definition}

\begin{lemma}\label{lem:monotonicity-properties-Nielsen-function}
For $x>0$, the following statements hold .
\begin{enumerate}[(i)]
\item $\beta(x)$ is decreasing.
\item $\beta^{(m)}(x)$ is positive and decreasing if $m$ is even.
\item $\beta^{(m)}(x)$ is negative and increasing if $m$ is odd.
\item $\left\lvert \beta^{(m)}(x) \right\rvert $ is decreasing for all $m\in \mathbb{N}$.
\end{enumerate}
\end{lemma}

\begin{proof}
These follow easily from ~(\ref{eqn:Nielsen-funt-Integral-Rep-2}) and ~(\ref{eqn:m-deriv-Nielsen-funt-Integral-Rep-2}).
\end{proof}

\begin{proposition}\label{pro:complete-monotonicity-Nielsen-funct}
The function $\beta(x)$ is completely monotonic. 
\end{proposition}

\begin{proof}
Let $x>0$ and $k\in \mathbb{N}_0 $. Then by ~(\ref{eqn:m-deriv-Nielsen-funt-Integral-Rep-2}) obtain
\begin{align*}
(-1)^{k} \beta^{(k)}(x)&=(-1)^{2k}\int_0^{\infty} \frac{t^{k}e^{-xt}}{1+e^{-t}}\,dt \geq0
\end{align*}
which completes the proof.
\end{proof}

\begin{remark}
More generally, $\beta^{(m)}(x)$ is completely monotonic if $m$ is even and $-\beta^{(m)}(x)$ is completely monotonic if $m$ is odd. To see this, note that for $x>0$ and $k,m\in \mathbb{N}_0$, we obtain
\begin{align*}
(-1)^{k} \beta^{(m+k)}(x)&=(-1)^{m+2k}\int_0^{\infty} \frac{t^{m+k}e^{-xt}}{1+e^{-t}}\,dt \geq (\leq)0
\end{align*}
respectively for even(odd) $m$.
\end{remark}

\begin{theorem}\label{thm:Double-Ineq-Nielsen-funct}
The double-inequality 
\begin{equation}\label{eqn:Double-Ineq-Nielsen-funct}
\frac{\beta'(a)+\beta'(b)}{2} \leq 
\frac{\beta(b) - \beta(a)}{b-a} \leq
\beta'\left( \frac{a+b}{2}\right)
\end{equation}
holds for $a,b>0$.
\end{theorem}

\begin{proof}
We employ the classical Hermite-Hadamard inequality which states that if $f:[a,b]\rightarrow \mathbb{R}$ is convex, then
\begin{equation*}
f\left(\frac{a+b}{2}\right) \leq
\frac{1}{b-a}\int_{a}^{b}f(x)\,dx  \leq
\frac{f(a)+f(b)}{2} .
\end{equation*}
Without loss of generality, let $b\geq a>0$ and $f(x)=-\beta'(x)$ for $x\in[a,b]$. Then $f(x)$ is convex and consequently, we obtain
\begin{equation*}
-\beta' \left(\frac{a+b}{2}\right) \leq
-\frac{1}{b-a}\int_{a}^{b}\beta'(x)\,dx  \leq
-\frac{\beta'(a)+\beta'(b)}{2} .
\end{equation*}
which gives the result ~(\ref{eqn:Double-Ineq-Nielsen-funct}). Alternatively, since $\beta'(x)$ is continuous and concave (i.e. $\beta'''(x)<0$) on $(0,\infty)$, then by Theorem 1 of \cite{Merkle-1998-AM}, we obtain the desired result.
\end{proof}

\begin{theorem}\label{thm:Ineq-mth-deriv-Nielsen-funct-via-Holder}
Let $m,n\in \mathbb{N}_0$, $a>1$, $\frac{1}{a}+\frac{1}{b}=1$ such that $\frac{m}{a}+\frac{n}{b}\in \mathbb{N}_0$. Then, the inequality
\begin{equation}\label{eqn:Ineq-mth-deriv-Nielsen-funct-via-Holder}
\left|  \beta^{(\frac{m}{a}+\frac{n}{b})}\left(\frac{x}{a}+\frac{y}{b}\right) \right|    \leq 
 \left| \beta^{(m)}(x) \right|^{\frac{1}{a}} 
 \left| \beta^{(n)}(y) \right|^{\frac{1}{b}}
\end{equation}
holds for $x,y>0$.
\end{theorem}

\begin{proof}
By the relation ~(\ref{eqn:m-deriv-Nielsen-funt-Integral-Rep-2}) and the H\"{o}lder's inequality,  we obtain
\begin{align*}
\left|  \beta^{(\frac{m}{a}+\frac{n}{b})}\left(\frac{x}{a}+\frac{y}{b}\right) \right|
&=\int_0^{\infty} \frac{t^{(\frac{m}{a}+\frac{n}{b})}e^{-(\frac{x}{a}+\frac{y}{b}) t}}{1+e^{-t}}\,dt  \\
&=\int_0^{\infty} \frac{t^{\frac{m}{a}}e^{-\frac{xt}{a}}}{(1+e^{-t})^{\frac{1}{a}}}.
\frac{t^{\frac{n}{b}}e^{-\frac{yt}{b}}}{(1+e^{-t})^{\frac{1}{b}}}\,dt  \\
&\leq \left(  \int_0^{\infty} \frac{t^{m}e^{-xt}}{1+e^{-t}}\,dt \right)^{\frac{1}{a}} 
 \left(\int_0^{\infty} \frac{t^{n}e^{-yt}}{1+e^{-t}}\,dt \right)^{\frac{1}{b}} \\
&= \left| \beta^{(m)}(x) \right|^{\frac{1}{a}}   \left| \beta^{(n)}(y) \right|^{\frac{1}{b}} 
\end{align*}
which completes the proof.
\end{proof}

\begin{remark}
Note that the absolute signs in ~(\ref{eqn:Ineq-mth-deriv-Nielsen-funct-via-Holder}) are not required if $m$ and $n$ are even.
\end{remark}

\begin{remark}
If $m=n$ is even in Theorem \ref{thm:Ineq-mth-deriv-Nielsen-funct-via-Holder}, then the inequality ~(\ref{eqn:Ineq-mth-deriv-Nielsen-funct-via-Holder}) becomes
\begin{equation}\label{eqn:Ineq-mth-deriv-Nielsen-funct-log-convexity}
 \beta^{(m)} \left(\frac{x}{a}+\frac{y}{b}\right)  \leq 
\left( \beta^{(m)}(x) \right)^{\frac{1}{a}} 
 \left( \beta^{(m)}(y) \right)^{\frac{1}{b}}
\end{equation}
which implies that the function $\beta^{(m)}(x)$ is logarithmically convex for even $m$. Moreover, if $m=0$ in ~(\ref{eqn:Ineq-mth-deriv-Nielsen-funct-log-convexity}), then we obtain
\begin{equation}\label{eqn:Ineq-Nielsen-funct-log-convexity}
 \beta \left(\frac{x}{a}+\frac{y}{b}\right)  \leq 
\left( \beta(x) \right)^{\frac{1}{a}} 
 \left( \beta(y) \right)^{\frac{1}{b}}
\end{equation}
implies that $\beta(x)$ is logarithmically convex.
\end{remark}

\begin{remark}
Let $a=b=2$, $x=y$ and $m=n+2$  in Theorem \ref{thm:Ineq-mth-deriv-Nielsen-funct-via-Holder}. Then we obain the Turan-type inequality
\begin{equation}\label{eqn:Turan-type-Nielsen-funct-abs}
\left|\beta^{(n+1)}(x)\right|^2 \leq  \left|\beta^{(n+2)}(x)\right| \left|\beta^{(n)}(x)\right| .
\end{equation}
Furthermore, if $n=0$ in ~(\ref{eqn:Turan-type-Nielsen-funct-abs}) then we  get
\begin{equation}\label{eqn:Turan-type-Nielsen-funct}
\left( \beta'(x) \right)^2 \leq \beta''(x)\beta(x) .
\end{equation}
\end{remark}

\begin{theorem}\label{thm:expo-x-times-Nielsen-funct}
Let $m\in \mathbb{N}_0$ be even. Then the function
\begin{equation}\label{eqn:expo-x-times-Nielsen-funct}
Q(x)=e^{ax}\beta^{(m)}(x)
\end{equation}
is convex for $x>0$ and any real number $a$.
\end{theorem}

\begin{proof}
Let $m$ be even and $a$ be any real number. Then for $x>0$,  
\begin{align*}
Q'(x)&= ae^{ax}\beta^{(m)}(x) + e^{ax}\beta^{(m+1)}(x) , \\
Q''(x)&= a^2e^{ax}\beta^{(m)}(x) + 2a e^{ax}\beta^{(m+1)}(x) + e^{ax}\beta^{(m+2)}(x)  \\
&= e^{ax} \left[ a^2\beta^{(m)}(x) + 2a\beta^{(m+1)}(x) + \beta^{(m+2)}(x) \right] .
\end{align*}
The quadratic function $f(a)=a^2\beta^{(m)}(x) + 2a\beta^{(m+1)}(x) + \beta^{(m+2)}(x)$ has a discriminant  $\Delta=4\left[ \left(\beta^{(m+1)}(x)\right)^2 - \beta^{(m)}(x)\beta^{(m+2)}(x) \right] \leq0$ as a result of ~(\ref{eqn:Turan-type-Nielsen-funct-abs}). Then, since $\beta^{(m)}(x)>0$, it follows that $f(a)\geq0$. Thus, $Q''(x)\geq0$ and this completes the proof.
\end{proof}

\begin{theorem}\label{thm:Nielsen-funct-to-power-a}
Let $m\in \mathbb{N}_0$ be even. Then the function
\begin{equation}\label{eqn:Nielsen-funct-to-power-a}
P(x)=\left[ \beta^{(m)}(x) \right]^{\alpha}
\end{equation}
is convex for $x>0$ and $\alpha>0$.
\end{theorem}

\begin{proof}
Let $m$ be even, $x>0$ and $\alpha>0$. Then
\begin{equation*}
\ln P(x)= \alpha \ln \beta^{(m)}(x) \quad \text{implies} \quad  \frac{P'(x)}{P(x)}=\alpha \frac{\beta^{(m+1)}(x)} {\beta^{(m)}(x)} .
\end{equation*}
That is,
\begin{equation*}
P'(x)=\alpha P(x) \frac{\beta^{(m+1)}(x)} {\beta^{(m)}(x)} 
\end{equation*}
and then
\begin{align*}
P''(x)&=P(x) \left\{ \left( \frac{P'(x)}{P(x)}\right)^2 + \alpha \left[ \frac{\beta^{(m+2)}(x)\beta^{(m)}(x) -(\beta^{(m+1)}(x))^2}{[\beta^{(m)}(x)]^2} \right]  \right\} \\
&\geq0
\end{align*}
as a result of ~(\ref{eqn:Turan-type-Nielsen-funct-abs}).
\end{proof}

\begin{theorem}\label{thm:Nielsen-funct-Ratio-Increasing}
Let $m\in \mathbb{N}_0$ be even. Then the function
\begin{equation}\label{eqn:Nielsen-funct-Ratio-Increasing}
U(x)=\frac{\beta^{(m)}(kx)}{\left[ \beta^{(m)}(x) \right]^k}
\end{equation}
is increasing if $k>1$ and decreasing if $0<k \leq1$.
\end{theorem}

\begin{proof}
For $x>0$ and $m$ even, define a function $S$ by 
\begin{equation*}
S(x)=\frac{\beta^{(m+1)}(x)}{\beta^{(m)}(x)} .
\end{equation*}
Then direct differentiation yields
\begin{equation*}
S'(x)=\frac{\beta^{(m+2)}(x)\beta^{(m)}(x)-\left(\beta^{(m+1)}(x)\right)^2}{\left[\beta^{(m)}(x)\right]^2}
\end{equation*}
and by ~(\ref{eqn:Turan-type-Nielsen-funct-abs}), we conclude that $S'(x)\geq0$. Hence $S(x)$ is increasing. Next, let $u(x)=\ln U(x)$. Then we obtain
\begin{equation*}
u'(x)= k \left[\frac{\beta^{(m+1)}(kx)}{\beta^{(m)}(kx)} - \frac{\beta^{(m+1)}(x)}{\beta^{(m)}(x)} \right] .
\end{equation*}
Since $S(x)$ is increasing, it follows that $u'(x)>0$ if $k>1$ and  $u'(x)\leq0$ if $0<k\leq1$. This completes the proof.
\end{proof}

\begin{corollary}\label{cor:Nielsen-funct-Ratio-Ineq}
Let $m\in \mathbb{N}_0$ be even and $0<x\leq y$. Then the inequality
\begin{equation}\label{eqn:Nielsen-funct-Ratio-Ineq}
\left( \frac{\beta^{(m)}(y)}{\beta^{(m)}(x)} \right)^k \leq \frac{\beta^{(m)}(ky)}{\beta^{(m)}(kx)} 
\end{equation}
is satisfied if $k>1$. It reverses if $0<k \leq1$.
\end{corollary}

\begin{proof}
This follows from the monotonicity property of $U(x)$ as defined in ~(\ref{eqn:Nielsen-funct-Ratio-Increasing}).
\end{proof}

\begin{theorem}\label{thm:Nielsen-funct-log-concave-and-Ineq}
Let $m\in \mathbb{N}_0$ be even and $a>0$. Then for $x>0$, the function
\begin{equation*}\label{eqn:Nielsen-funct-log-concavity}
\Omega(x)=\frac{\beta^{(m)}(a)}{\beta^{(m)}(x+a)}
\end{equation*}
is increasing and logarithmically concave, and  the inequality
\begin{equation}\label{eqn:Nielsen-funct-log-concavity-Ineq}
1<\frac{\beta^{(m)}(a)}{\beta^{(m)}(x+a)}
\end{equation}
is satisfied.
\end{theorem}

\begin{proof}
Define $\mu$ for $m\in \mathbb{N}_0$ even, $a>0$ and $x>0$ by 
\begin{equation*}
\mu(x) = \ln \Omega(x) = \ln \beta^{(m)}(a) - \ln \beta^{(m)}(x+a) .
\end{equation*}
Then 
\begin{equation*}
\mu'(x) = -\frac{\beta^{(m+1)}(x+a)}{\beta^{(m)}(x+a)}>0
\end{equation*}
which implies that $\mu(x)$ in increasing. Consequently, $\Omega(x)=e^{\mu(x)}$ is  increasing. Next, we have
\begin{equation*}
\left( \ln \Omega(x) \right)'' = - \left[ \frac{\beta^{(m+2)}(x+a)\beta^{(m)}(x+a) -(\beta^{(m+1)}(x+a))^2}{[\beta^{(m)}(x+a)]^2} \right] \leq0
\end{equation*}
which implies that $ \Omega(x)$ is logarithmically concave. Furthermore,
\begin{equation*}
\lim_{x \rightarrow 0^+}\Omega(x)=1 \quad \text{and} \quad  \lim_{x \rightarrow \infty}\Omega(x) =\infty .
\end{equation*}
Then since $\Omega(x)$ is increasing, we obtain the result ~(\ref{eqn:Nielsen-funct-log-concavity-Ineq}). 
\end{proof}

\begin{theorem}\label{thm:sub-super-additivity-Nielsen-funct}
Let $m\in \mathbb{N}_0$. Then the following inequalities hold for $x,y>0$.
\begin{equation}\label{eqn:sub-super-additivity-even}
\beta^{(m)}(x+y) \leq \beta^{(m)}(x) + \beta^{(m)}(y)
\end{equation}
if $m$ is even, and
\begin{equation}\label{eqn:sub-super-additivity-odd}
\beta^{(m)}(x+y) \geq \beta^{(m)}(x) + \beta^{(m)}(y)
\end{equation}
if $m$ is odd.
\end{theorem}

\begin{proof}
Let $m$ be even and $H(x)=\beta^{(m)}(x+y) - \beta^{(m)}(x) - \beta^{(m)}(y)$. Then for a fixed $y$, we obtain
\begin{align*}
H'(x)&=\beta^{(m+1)}(x+y) - \beta^{(m+1)}(x)  \\
&= (-1)^{(m+1)} \int_{0}^{\infty} \frac{t^{m}\left( e^{-(x+y)t}-e^{-xt}\right)}{1+e^{-t}}\,dt  \\
&= - \int_{0}^{\infty} \frac{t^{m}}{1+e^{-t}}\left( e^{-yt} - 1\right)\,dt  \\
&\geq0.
\end{align*}
Hence, $H(x)$ is increasing. Moreover,
\begin{align*}
\lim_{x \rightarrow \infty}H(x)&=\lim_{x \rightarrow \infty}\left[ \beta^{(m)}(x+y) - \beta^{(m)}(x) - \beta^{(m)}(y) \right] \\
&= (-1)^{m} \lim_{x \rightarrow \infty}\left[ \int_{0}^{\infty} \frac{t^{m}}{1+e^{-t}}\left( e^{-(x+y)t} - e^{-xt} - e^{-yt}  \right)\,dt \right] \\
&= - \int_{0}^{\infty} \frac{t^{m}e^{-yt}}{1+e^{-t}}\,dt\\
&\leq0.
\end{align*}
Therefore, $H(x)\leq0$ which gives the result ~(\ref{eqn:sub-super-additivity-even}). Similarly, for $m$ odd,  we obtain $H'(x)\leq0$ and $\lim_{x \rightarrow \infty}H(x)\geq0$ which implies that $H(x)>0$ and this gives the result ~(\ref{eqn:sub-super-additivity-odd}). 
\end{proof}

\begin{remark}
Theorem \ref{thm:sub-super-additivity-Nielsen-funct} is another way of saying that the function $\beta^{(m)}(x)$ is subadditive if $m$ is even, and superadditive if $m$ is odd.
\end{remark}

\begin{theorem}\label{thm:starshaped-Nielsen-funct}
Let $m\in \mathbb{N}_0$. Then for $m$ odd, the function $\beta^{(m)}(x)$ is star-shaped on $(0,\infty)$. That is, 
\begin{equation}\label{eqn:starshaped-Nielsen-funct}
\beta^{(m)}(\alpha x) \leq \alpha \beta^{(m)}(x) 
\end{equation}
for all $x\in(0,\infty)$ and $\alpha \in(0,1]$.
\end{theorem}

\begin{proof}
Let $m$ be odd and $T(x)=\beta^{(m)}(\alpha x) -\alpha \beta^{(m)}(x)$. Then for $x\in(0,\infty)$ and $\alpha \in(0,1]$, we have
\begin{align*}
T'(x)&= \alpha \left[ \beta^{(m+1)}(\alpha x) -  \beta^{(m+1)}(x) \right] \\
&\geq0.
\end{align*}
Thus, $T(x)$ is increasing. Recall that $\beta^{(n)}(x)$ is decreasing for even $n$. Then since $0<\alpha x\leq x$, we have $\beta^{(m+1)}(\alpha x) \geq \beta^{(m+1)}(x)$.   Furthermore,
\begin{align*}
\lim_{x \rightarrow \infty}T(x)&=\lim_{x \rightarrow \infty}\left[\beta^{(m)}(\alpha x) -\alpha \beta^{(m)}(x)\right] \\
&=  \lim_{x \rightarrow \infty}\left[\int_{0}^{\infty} \frac{t^{m}e^{-\alpha xt}}{1+e^{-t}}\,dt - 
\alpha \int_{0}^{\infty} \frac{t^{m}e^{-xt}}{1+e^{-t}}\,dt \right] \\
&=0.
\end{align*}
Therefore, $T(x)\leq0$ which completes the proof.
\end{proof}

\begin{theorem}\label{thm:Ineq-Nielsen-funct-Square-Product}
Let $m\in \mathbb{N}_0$. Then the inequality
\begin{equation}\label{eqn:Ineq-Nielsen-funct-Square-Product}
\left[\beta^{(m)}(xy)\right]^2 \leq \beta^{(m)}(x) \beta^{(m)}(y) 
\end{equation}
holds for $x\geq1$ and $y\geq1$.
\end{theorem}

\begin{proof}
We have $xy\geq x$ and $xy\geq y$ since $x\geq 1$ and $y\geq1$. If $m$ is even, then we obtain
\begin{equation*}
0<\beta^{(m)}(xy) \leq \beta^{(m)}(x)
\end{equation*}
and
\begin{equation*}
0<\beta^{(m)}(xy) \leq \beta^{(m)}(y)
\end{equation*}
since $\beta^{(m)}(x)$ is decreasing for even $m$ (see Lemma \ref{lem:monotonicity-properties-Nielsen-function}). That implies
\begin{equation*}
\left[\beta^{(m)}(xy)\right]^2 \leq \beta^{(m)}(x) \beta^{(m)}(y) .
\end{equation*}
Also, if $m$ is odd, then we have
\begin{equation*}
0>\beta^{(m)}(xy) \geq \beta^{(m)}(x)
\end{equation*}
and
\begin{equation*}
0>\beta^{(m)}(xy) \geq \beta^{(m)}(y)
\end{equation*}
since $\beta^{(m)}(x)$ is increasing for odd $m$, and that also implies
\begin{equation*}
\left[\beta^{(m)}(xy)\right]^2 \leq \beta^{(m)}(x) \beta^{(m)}(y) 
\end{equation*}
which completes the proof.
\end{proof}

\noindent
A generalization of Theorem \ref{thm:Ineq-Nielsen-funct-Square-Product} is given as follows.

\begin{theorem}\label{thm:Ineq-Nielsen-funct-Square-Product-Gen}
Let $n\in \mathbb{N}$ and $m\in \mathbb{N}_0$ such that $m$ is even. Then the inequality
\begin{equation}\label{eqn:Ineq-Nielsen-funct-Square-Product-Gen}
\beta^{(m)}\left(\prod_{i=1}^{n}x_i \right) \leq 
\left( \prod_{i=1}^{n} \beta^{(m)}(x_i) \right)^{\frac{1}{n}}
\end{equation}
holds for $x_i\geq1$, $i=1,2,3\dots,n$.
\end{theorem}

\begin{proof}
Since $x_i\geq1$ for $i=1,2,3\dots,n$, we have $\prod_{i=1}^{n}x_i \geq x_j$ for $j=1,2,3\dots,n$. For $m$ even, we have  
\begin{align*}
0<\beta^{(m)}\left(\prod_{i=1}^{n}x_i \right) & \leq \beta^{(m)}(x_1),  \\
0<\beta^{(m)}\left(\prod_{i=1}^{n}x_i \right) & \leq \beta^{(m)}(x_2),  \\
\vdots \qquad & \qquad  \qquad   \vdots \\
0<\beta^{(m)}\left(\prod_{i=1}^{n}x_i \right) & \leq \beta^{(m)}(x_n).
\end{align*}
Then by taking products of these inequalities, we obtain
\begin{equation*}
\left[ \beta^{(m)}\left(\prod_{i=1}^{n}x_i \right) \right]^n \leq 
\prod_{i=1}^{n} \beta^{(m)}(x_i) 
\end{equation*}
which completes the proof.
\end{proof}

\begin{theorem}\label{thm:ineq-mth-deriv-Nielsen-funct-via-Minkowski}
Let $m,n\in \mathbb{N}_0$ and $s\geq1$. Then, the inequality
\begin{equation}\label{eqn:ineq-mth-deriv-Nielsen-funct-via-Minkowski}
\left( \left| \beta^{(m)}(x)\right|  +  \left| \beta^{(n)}(y)\right| \right)^{\frac{1}{s}} \leq 
 \left| \beta^{(m)}(x)\right|^{\frac{1}{s}}  + \left| \beta^{(n)}(y)\right|^{\frac{1}{s}} 
\end{equation}
holds for $x,y>0$.
\end{theorem}

\begin{proof}
Note that  $u^s+v^s \leq (u+v)^s$, for $u,v\geq0$ and  $s\geq1$. Then by the Minkowski's inequality, we obtain
\begin{align*}
\left( \left| \beta^{(m)}(x)\right|  +  \left| \beta^{(n)}(y)\right| \right)^{\frac{1}{s}}
&=\left( \int_{0}^{\infty} \frac{t^{m}e^{-xt}}{1+e^{-t}}\,dt + \int_{0}^{\infty} \frac{t^{n}e^{-yt}}{1+e^{-t}}\,dt   \right)^{\frac{1}{s}} \\
&=\left( \int_{0}^{\infty} \left[ 
\left(\frac{t^{\frac{m}{s}}e^{\frac{-xt}{s}}}{(1+e^{-t})^{\frac{1}{s}}}\right)^s +  \left(\frac{t^{\frac{n}{s}}e^{\frac{-yt}{s}}}{(1+e^{-t})^{\frac{1}{s}}}\right)^s 
\right]\,dt  \right)^{\frac{1}{s}} \\
&\leq \left( \int_{0}^{\infty} \left[ 
\left(\frac{t^{\frac{m}{s}}e^{\frac{-xt}{s}}}{(1+e^{-t})^{\frac{1}{s}}}\right) +  \left(\frac{t^{\frac{n}{s}}e^{\frac{-yt}{s}}}{(1+e^{-t})^{\frac{1}{s}}}\right) 
\right]^s\,dt  \right)^{\frac{1}{s}} \\
&\leq \left( \int_{0}^{\infty} \frac{t^{m}e^{-xt}}{1+e^{-t}}\,dt \right)^{\frac{1}{s}}  +
         \left( \int_{0}^{\infty} \frac{t^{n}e^{-yt}}{1+e^{-t}}\,dt \right)^{\frac{1}{s}} \\
&= \left| \beta^{(m)}(x)\right|^{\frac{1}{s}}  + \left| \beta^{(n)}(y)\right|^{\frac{1}{s}} 
\end{align*}
which yields the desired result.
\end{proof}

\begin{remark}\label{rem:Funct-Eqn-Nielsen-Abs}
Notice that $\left| \beta^{(m)}(x)  \right| = (-1)^{m}\beta^{(m)}(x)$ for $m\in \mathbb{N}_0$ and $x>0$. Then by the recurrence relation ~(\ref{eqn:Nielsen-funct-Gen-Funct-Eqn}), we obtain
\begin{equation}\label{eqn:Funct-Eqn-Nielsen-Abs}
\left| \beta^{(m)}(x+1) \right| =  \frac{m!}{x^{m+1}} - \left| \beta^{(m)}(x) \right|
\end{equation}
which implies 
\begin{equation}\label{eqn:Ineq-mth-deriv-Nielsen-Abs}
\left| \beta^{(m)}(x) \right| \leq  \frac{m!}{x^{m+1}} .
\end{equation}
\end{remark}

\begin{theorem}\label{thm:ineq-mth-deriv-Nielsen-funct-via-MVT}
Let $m\in \mathbb{N}_0$ and $0<a<b$. Then, there exists a $\lambda \in (a,b)$ such that
\begin{equation}\label{eqn:ineq-mth-deriv-Nielsen-funct-via-MVT}
\left| \beta^{(m)}(b) - \beta^{(m)}(a) \right|  \leq (b-a) \frac{(m+1)!}{\lambda^{m+2}}
\end{equation}
\end{theorem}

\begin{proof}
By the classical mean value theorem, there exist a $\lambda \in (a,b)$ such that $\frac{\beta^{(m)}(b)  - \beta^{(m)}(a) }{b-a} = \beta^{(m+1)}(\lambda)$. Thus, $\frac{\left| \beta^{(m)}(b)  - \beta^{(m)}(a) \right|}{(b-a)} = \left| \beta^{(m+1)}(\lambda) \right|$  and by ~(\ref{eqn:Ineq-mth-deriv-Nielsen-Abs}), we obtain the result ~(\ref{eqn:ineq-mth-deriv-Nielsen-funct-via-MVT}). 
\end{proof}

\section{Conclusion}

\noindent
In this study, we obtained some convexity, monotonicity and additivity properties as well as some inequalities involving the Nielsen's $\beta$-function.  The established results may be useful in evaluating or estimating certain integrals. Furthermore, the findings could provide useful information for further study of the function.

\section*{Conflict of Interests}
\noindent
The author declares that there is no conflict of interests regarding the publication of this paper.

\bibliographystyle{plain}


\end{document}